\documentclass[12pt]{article}

\usepackage[margin=1in]{geometry}  
\usepackage{graphicx}              
\usepackage{amsmath}               
\usepackage{amsfonts}              
\usepackage{amsthm}                

\usepackage{authblk}
\usepackage{blindtext}
\usepackage{geometry,amssymb,graphicx,tensor,epigraph,amsmath,amsthm, xspace, listings, stmaryrd, float, tikz-cd, mathtools, afterpage}
\usepackage[T1]{fontenc}
\usepackage[utf8]{inputenc}
\usepackage[english]{babel}
\usepackage{bm}
\usepackage{hyperref}
\hypersetup{
    colorlinks=true,
    linkcolor=black,
    filecolor=magenta,
    urlcolor=blue,
    pdftitle={An Extension of Major-Minor Mean Field Game Theory to Joint State-Control Measures},
    pdfpagemode=FullScreen,
    }

\newtheorem{teo}{Theorem}[section]
\newtheorem{lema}[teo]{Lemma}
\newtheorem{prop}[teo]{Proposition}
\newtheorem{cor}[teo]{Corollary}

\newtheorem*{obs*}{Remark}

\newtheorem*{lema*}{Lemma}

\theoremstyle{definition}

\newcommand{\keywords}[1]{\textbf{\textit{Keywords---}} #1}

\newcommand{\agus}[1]{}

\begin{document}

\title{An Extension of Major-Minor Mean Field Game Theory}

\author[1]{Agust\'in Mu\~noz Gonz\'alez}

\affil[1]{Departamento de Matem\'aticas, Facultad de Ciencias Exactas y Naturales, Universidad de Buenos Aires, Buenos Aires, Argentina}

\maketitle


\begin{abstract}
This work extends the theory presented in \textit{Mean Field Games with a Dominating Player} by Bensoussan, Chau and Yam \cite{bensoussan2014mfgdominating} on mean field games with a dominating player, to the case in which the utility and cost functions depend not only on the law of the states, but on the joint state--control law. We incorporate the unified notation $\Pi_t = \mathcal{L}(X_t^1, U_t^1 \mid \mathcal{F}_t^0) \in \mathcal{P}_2(\mathbb{R}^{n_1} \times A)$, which describes the conditional distribution of the state--control pair of the representative agent given the common noise of the dominating player. In addition, we generalize the role of the dominating player to include the direct impact of its controls $u_0$ on the dynamics and functionals of the system. The optimization problems are reformulated in terms of $\Pi_t$, the necessary optimality conditions are established via stochastic maximum principles, and a coupled SHJB--FP system of equations is obtained that synthesizes the equilibrium conditions. This framework provides a significant extension of the existing literature on MFG with a dominating player.
\end{abstract}
\hspace{10pt}
\clearpage
\tableofcontents
\clearpage
\section{Introduction}

In this section we extend the framework of \textit{mean field games with a dominating player} (major--minor mean field games), developed in \cite{bensoussan2014mfgdominating}, to the case in which the utility and cost functions depend not only on the law of the states, but on the joint state--control law. This extension allows us to capture phenomena where the interaction among representative agents is not described solely through their state trajectories, but also through the distributions induced by their control decisions.

In addition, we generalize the role of the dominating player: in the classical model, its influence on the system appears only through its state \(x_0\); here we also consider the direct impact of its controls \(u_0\). In this way, the extended framework encompasses both dependencies via the dominating state and dependencies via its actions, a crucial aspect in applications where the decisions of the major player determine the dynamics and rewards of the rest of the system.

We introduce the notation
\[
\Pi_t = \mathcal{L}\big(X_t^1, U_t^1 \,\big|\, \mathcal{F}_t^0\big)\in\mathcal{P}_2(\mathbb{R}^{n_1}\times A),
\]
which describes the conditional distribution of the pair \((x_1,u_1)\) of the representative agent given the common noise of the dominating player. As special cases, the law of the states and the law of the controls are obtained as marginals:
\[
\mu_t = (\mathrm{pr}_X)_{\#}\Pi_t,
\qquad
q_t = (\mathrm{pr}_U)_{\#}\Pi_t.
\]

The advantage of working with $\Pi_t$ is that it unifies both dependencies in a single measure on the product space, which simplifies both the notation and the application of calculus on the space of probability measures (Lions derivatives).

In the MFG literature it is customary for the mean field interaction to occur through the law of the states, $\mu$, and not through the law of the controls. Some authors, such as R. Carmona and D. Lacker, analyze this case by introducing the law of the controls into the cost function optimized by the agents \cite{carmona2015weakmfgformulation}. However, when a new player with dominating characteristics is introduced, the results found in the associated literature refer to the classical MFG setting. For example, in \cite{bensoussan2014mfgdominating} a game is analyzed with one dominating agent and many small agents, where the functionals depend only on the law of the states:

\begin{align*}
    J_0(u_0) &= \mathbb{E} \left[ \int_0^T f_0 \left( x_0(t), \mu(t), u_0(t) \right) dt + h_0 \left( x_0(T), \mu(T) \right) \right], \\
    J_1(u_1, x_0, \mu) &= \mathbb{E} \left[ \int_0^T f_1 \left( x_1(t), x_0(t), \mu(t), u_1(t) \right) dt + h_1 \left( x_1(T), x_0(T), \mu(T) \right) \right].
\end{align*}

In this work, we adapt the theory of \cite{bensoussan2014mfgdominating} to the case in which the mean field interactions occur through the joint state--control law $\Pi_t$.

The objectives of this work are:
\begin{enumerate}
    \item To reformulate the optimization problems of the representative agent and the dominating player in terms of $\Pi$ (Problems 1--3).
    \item To recall the necessary technical notions: diffusion operators and Fokker--Planck equations, as well as the Lions derivative for functions depending on joint measures.
    \item To extend the results on necessary optimality conditions to the state--control case, stating and proving the corresponding lemmas (analogues of Lemmas 24--26 in \cite{bensoussan2014mfgdominating}).
\end{enumerate}

The proofs follow the same strategy as in the classical case: variations of the controls are considered, the variations of the objective functions are expressed in terms of the measure $\Pi$ and the induced dynamics, and via duality adjoint processes are introduced that allow the minimizers to be characterized in terms of stochastic Hamilton--Jacobi--Bellman (SHJB) problems coupled with Fokker--Planck equations.

\clearpage

\section{Working spaces and assumptions}

We work on a filtered probability space
\[
(\Omega,\mathcal{F},\mathbb{P};\,\mathbb{F}^0=(\mathcal{F}^0_t)_{t\in[0,T]},\,\mathbb{F}^1=(\mathcal{F}^1_t)_{t\in[0,T]}),
\]
with two independent adapted Brownian motions $W^0$ and $W^1$ (of dimensions $d_0$ and $d_1$), and initial conditions $\xi_0\in \mathbb{R}^{n_0}$, $\xi_1\in \mathbb{R}^{n_1}$ independent of $(W^0,W^1)$.

Let \( x_0(t) \in \mathbb{R}^{n_0} \) and \( x_1(t) \in \mathbb{R}^{n_1} \) denote the state processes for the dominating player and a representative agent, respectively, whose dynamics are given by the following stochastic differential equations:

\[
\begin{aligned}
    dx_0 &= g_0 \left( x_0(t), \Pi_t, u_0(x_0(t), t) \right) dt + \sigma_0 \left( x_0(t) \right) dW_0(t), \\
    x_0(0) &= \xi_0, \\
    dx_1 &= g_1 \left( x_1(t), x_0(t), u_0(t), \Pi_t, u_1(x_1(t)) \right) dt + \sigma_1 \left( x_1(t) \right) dW_1(t), \\
    x_1(0) &= \xi_1.
\end{aligned}
\]

The functional coefficients are defined as follows:

\[
\begin{aligned}
    g_0 &: \mathbb{R}^{n_0} \times \mathcal{P}_2(\mathbb{R}^{n_1}\times A) \times \mathbb{R}^{m_0} \to \mathbb{R}^{n_0}, \\
    g_1 &: \mathbb{R}^{n_1} \times \mathbb{R}^{n_0} \times \mathbb{R}^{m_0} \times \mathcal{P}_2(\mathbb{R}^{n_1}\times A) \times \mathbb{R}^{m_1} \to \mathbb{R}^{n_1}, \\
    \sigma_0 &: \mathbb{R}^{n_0} \to \mathbb{R}^{n_0 \times d_0}, \\
    \sigma_1 &: \mathbb{R}^{n_1} \to \mathbb{R}^{n_1 \times d_1}.
\end{aligned}
\]

The dominating player and the representative agents also have the following objective functionals, respectively:
\begin{equation}
\label{eq:J0}
    J_0(u_0) = \mathbb{E} \left[ \int_0^T f_0 \left( x_0(t), \Pi_t, u_0(t) \right) dt + h_0 \left( x_0(T), \Pi_T \right) \right],
\end{equation}
\begin{equation}
\label{eq:J1}
    J_1(u_1, x_0, u_0, \Pi) = \mathbb{E} \left[ \int_0^T f_1 \left( x_1(t), x_0(t), u_0(t), \Pi_t, u_1(t) \right) dt + h_1 \left( x_1(T), x_0(T), u_0(T), \Pi_T \right) \right].
\end{equation}
The functions are defined as follows:

\[
\begin{aligned}
    f_0 &: \mathbb{R}^{n_0} \times \mathcal{P}_2(\mathbb{R}^{n_1}\times A) \times \mathbb{R}^{m_0} \to \mathbb{R}, \\
    f_1 &: \mathbb{R}^{n_1} \times \mathbb{R}^{n_0} \times \mathbb{R}^{m_0} \times \mathcal{P}_2(\mathbb{R}^{n_1}\times A) \times \mathbb{R}^{m_1} \to \mathbb{R}, \\
    h_0 &: \mathbb{R}^{n_0} \times \mathcal{P}_2(\mathbb{R}^{n_1}\times A) \to \mathbb{R}, \\
    h_1 &: \mathbb{R}^{n_1} \times \mathbb{R}^{n_0} \times \mathbb{R}^{m_0} \times \mathcal{P}_2(\mathbb{R}^{n_1}\times A) \to \mathbb{R}.
\end{aligned}
\]

Here, \( u_0 \in \mathbb{R}^{m_0} \) and \( u_1 \in \mathbb{R}^{m_1} \) represent the respective controls of the dominating player and the representative agents.

To guarantee the existence of optimal controls and the equilibrium of the game, we impose the following assumptions on the functions defining the state processes and payoffs. Unlike the original formulation of Bensoussan, Chau and Yam \cite{bensoussan2014mfgdominating}, where the functions depend on the law of the states $\mu_t\in\mathcal{P}(\mathbb{R}^{n_1})$, in our extension the functions depend on the joint state--control law $\Pi_t\in\mathcal{P}(\mathbb{R}^{n_1}\times A)$. The assumptions are adapted accordingly.

\begin{itemize}

\item (A.1) Lipschitz continuity:

The functions \( g_0 \), \( \sigma_0 \), \( g_1 \) and \( \sigma_1 \) are globally Lipschitz continuous in all their arguments. In particular, there exists a constant \( K > 0 \) such that:

\begin{align*}
|g_0(x_0, \Pi, u_0) - g_0(x_0', \Pi', u_0')| &\leq K \left( |x_0 - x_0'| + W_2(\Pi, \Pi') + |u_0 - u_0'| \right);\\
|\sigma_0(x_0) - \sigma_0(x_0')| &\leq K |x_0 - x_0'|;\\
|g_1(x_1, x_0, u_0, \Pi, u_1) - g_1(x_1', x_0', u_0', \Pi', u_1')| &\leq K \left( |x_1 - x_1'| + |x_0 - x_0'| + |u_0 - u_0'| \right) \\
& + K\left(W_2(\Pi, \Pi') + |u_1 - u_1'| \right);\\
|\sigma_1(x_1) - \sigma_1(x_1')| &\leq K |x_1 - x_1'|.
\end{align*}

\item (A.2) Linear growth:

The functions \( g_0 \), \( \sigma_0 \), \( g_1 \) and \( \sigma_1 \) grow at most linearly in all their arguments. In particular, there exists a constant \( K > 0 \) such that:

\begin{align*}
|g_0(x_0, \Pi, u_0)| &\leq K(1 + |x_0| + M_2(\Pi) + |u_0|);\\
|\sigma_0(x_0)| &\leq K(1 + |x_0|);\\
|g_1(x_1, x_0, u_0, \Pi, u_1)| &\leq K(1 + |x_0| + |x_1| + |u_0| + M_2(\Pi) + |u_1|);\\
|\sigma_1(x_1)| &\leq K(1 + |x_1|).
\end{align*}

\item (A.3) Quadratic growth condition on the cost function:

There exists a constant \( K > 0 \) such that:

\begin{align*}
|f_1(x_1, x_0, u_0, \Pi, u_1) - f_1(x_1', x_0', u_0', \Pi', u_1')| &\leq K \left( 1 + |x_1| + |x_1'| + |x_0| + |x_0'| + |u_0| + |u_0'| \right. \\
&+ \left. M_2(\Pi) + M_2(\Pi') + |u_1| + |u_1'| \right) \\
&\cdot \left( |x_1 - x_1'| + |x_0 - x_0'| + |u_0 - u_0'|\right. \\
&+\left. W_2(\Pi, \Pi') + |u_1 - u_1'| \right);\\
|h_1(x_1, x_0, u_0, \Pi) - h_1(x_1', x_0', u_0', \Pi')| &\leq K \left( 1 + |x_1| + |x_1'| + |x_0| + |x_0'| + |u_0| + |u_0'| \right. \\
&+\left. M_2(\Pi) + M_2(\Pi') \right)\\
&\cdot \left( |x_1 - x_1'| + |x_0 - x_0'| + |u_0 - u_0'| + W_2(\Pi, \Pi') \right).
\end{align*}

\item (A.4) Differentiability:

The functions \( g_0 \), \( f_0 \), \( h_0 \), \( g_1 \), \( f_1 \) and \( h_1 \) are continuously differentiable in \( x_0 \in \mathbb{R}^{n_0} \), \( x_1 \in \mathbb{R}^{n_1} \), \( u_0 \in \mathbb{R}^{m_0} \), \( u_1 \in \mathbb{R}^{m_1} \) with bounded derivatives. We denote, for example, the derivative of \( g_0 \) with respect to \( x_0 \) by \( g_{0,x_0} \). They are also G\^{a}teaux differentiable in \( \Pi = m d\lambda \), for example, for \( m \in \mathcal{L}^2(\mathbb{R}^{n_1}) \):

\[
\frac{d}{d\theta} \bigg|_{\theta = 0} g_0(x_i, (m + \theta \tilde{m}) d\lambda, u_i) = \int_{\mathbb{R}^n} \frac{\partial g_0}{\partial m}(x_i, m d\lambda, u_i)(\xi) \tilde{m}(\xi) d\xi.
\]

\item (A.5):

The functions \( \sigma_0 \) (resp. \( \sigma_1 \)) are twice continuously differentiable in \( x_0 \) (resp. \( x_1 \)) with bounded first and second order derivatives.

\end{itemize}

\section{Preliminaries}

\subsection{Derivatives on the space of measures: Lions, linear, and Wasserstein gradient}

Let $\mathcal{P}_2(\mathbb{R}^d)$ be the space of probability measures with finite second moment. There are three equivalent notions (under regularity) of the derivative of functionals $f:\mathcal{P}_2(\mathbb{R}^d)\to\mathbb{R}$ that we will use as reference. We follow the notes of \cite{guo2022itoformula, korici2024lionsderivative}.

\paragraph{Lions derivative (via \emph{lift}).}
One identifies $\mu\in\mathcal{P}_2$ with a random variable $\vartheta\in L^2(\Omega;\mathbb{R}^d)$ such that $\mathbb{P}_\vartheta=\mu$, and \emph{lifts} $f$ to $\tilde f:L^2\to\mathbb{R}$ by $\tilde f(\vartheta)=f(\mathbb{P}_\vartheta)$. We say that $f$ is differentiable at $\mu_0$ if $\tilde f$ is Fr\'{e}chet differentiable at some $\vartheta_0$ with law $\mu_0$, and then there exists $h_0:\mathbb{R}^d\to\mathbb{R}^d$ such that $D\tilde f(\vartheta_0)=h_0(\vartheta_0)$. The \emph{Lions gradient} of $f$ at $\mu_0$ is the Borel version $(\partial_\mu f)(\mu_0,\cdot):=h_0(\cdot)$.

\paragraph{Linear derivative.}
The linear derivative $\frac{\delta f}{\delta\mu}(\mu,\cdot)$ is (defined up to an additive constant) the function that represents the variation along the geodesic segment $\mu_h=h\mu+(1-h)\mu'$, via
\[
f(\mu)-f(\mu')=\int_0^1\int_{\mathbb{R}^d}\frac{\delta f}{\delta\mu}(\mu_h,x)\,(\mu-\mu')(dx)\,dh.
\]
Under regularity, the two notions are related by
\[
\partial_\mu f(\mu,x)\;=\;\partial_x\Big(\tfrac{\delta f}{\delta\mu}(\mu,x)\Big).
\]

Recall that $W_2$ denotes the Wasserstein metric of order $2$ on $\mathcal{P}_2(\mathbb{R}^d)$, defined by
\[
W_2(\mu,\nu)^2 \;=\; \inf_{\pi \in \Gamma(\mu,\nu)} \int_{\mathbb{R}^d\times \mathbb{R}^d} |x-y|^2 \, d\pi(x,y),
\]
where $\Gamma(\mu,\nu)$ is the set of couplings of $\mu$ and $\nu$.

\paragraph{Wasserstein gradient (intrinsic).}
Another notion (from optimal transport) defines the differentiability of $f$ at $\mu$ by the existence of a vector $\xi\in T_\mu\mathcal{P}_2$ such that
\[
\frac{f(\mu_n)-f(\mu)-\int \xi(x)\,(y-x)\,\pi_n(dx,dy)}{W_2(\mu_n,\mu)}\;\to\;0
\]
for $\mu_n\to\mu$ in $W_2$ and $\pi_n$ optimal transport plans. This $\xi$ is denoted $\nabla_W f(\mu)$ and is equivalent to $\partial_\mu f(\mu,\cdot)$.

\begin{obs*}
We use the Lions derivative since it is especially convenient because:
\begin{enumerate}
    \item It is formulated in the Hilbert space $L^2$, which facilitates chain rules, BSDEs, and the Dynamic Programming Principle;
    \item It directly admits an It\^{o} formula over flows of laws;
    \item It is explicitly related to the other notions (linear and $\nabla_W$), so no generality is lost.
\end{enumerate}
\end{obs*}

\paragraph{Notation for joint state--control laws.}
In our extension we work on the product space with $\Pi_t\in\mathcal{P}_2(\mathbb{R}^{n_1}\times A)$ and disintegration $\Pi_t(dx,du)=m_t(dx)\,\Lambda_x(du)$. We define the \emph{joint Lions derivative} $D_\Pi f(\Pi)(x,u)$ as the Lions gradient of $f$ on the product space. To align with the rest of the work, we use the shorthand
\[
\big\langle D_\Pi f(\Pi),\delta_x\big\rangle \;:=\; \int_A D_\Pi f(\Pi)(x,u)\,\Lambda_x(du),
\]
which coincides with the projection onto the state when $f$ depends on $\Pi$ only through $m_t$. If moreover $f(\Pi)=F(m_t,q_t)$, the projections of $D_\Pi f$ recover the marginal derivatives $\partial_m F$ and the dependence on $q_t$. This convention simplifies the writing of the optimality conditions and the SHJB--FP systems.

\subsection{Fokker--Planck equation}

In general, when working with probability measures and comparing them, it is more convenient to do so when they possess density functions on $\mathbb{R}^{n_1}$. We define the second-order operator $A_1$ and its adjoint $A_1^*$ by
\begin{equation}
    \begin{aligned}
    A_1\Phi(x,t) &= -\mathrm{tr}(a_1(x)D^2\Phi(x,t)), \\
    A_1^*\Phi(x,t) &= -\sum_{i,j=1}^{n_1}\frac{\partial^2}{\partial x_i \partial x_j}\left(a_1^{ij}(x)\Phi(x,t)\right),
    \end{aligned}
\end{equation}
where $a_1(x)=\frac{1}{2}\sigma_1(x)\sigma_1(x)^*$ is a positive definite matrix.

Let $x_1(t)$ be the solution of the SDE of the representative agent under a control $u_1$, and suppose that the conditional density $p_{u_1}(\cdot,t)$ of $x_1^{u_1}(t)$ exists. Then $p_{u_1}$ satisfies the Fokker--Planck equation with dependence on the joint law $\Pi$,
\begin{equation}
    \label{Fokker-Planck}
    \begin{cases}
        \begin{aligned}
            \partial_t p_{u_1}(x,t) &= -A_1^* p_{u_1}(x,t)
            - \mathrm{div}\!\left(g_1\!\left(x,x_0(t),u_0(t), \Pi_t,u_1(x,t)\right)p_{u_1}(x,t)\right),\\[0.2cm]
            p_{u_1}(x,0) &= \omega(x),
        \end{aligned}
    \end{cases}
\end{equation}
where $\omega$ is the initial density of $\xi_1$. We further assume that $p_{u_1}(\cdot,t)\in L^2(\mathbb{R}^{n_1})$ and that $p_{u_1}(\cdot,t)\,d\lambda \in \mathcal{P}_2(\mathbb{R}^{n_1})$. For any density $m$ we write $m\,d\lambda= m$ when the context admits no ambiguity.

\medskip

\noindent
We can then rewrite the cost functional of the representative agent as

\begin{align*}
    J_1(u_1, x_0, u_0,\Pi)
    &= \mathbb{E}\!\left[ \int_0^T\!\int_{\mathbb{R}^{n_1}}
    f_1\!\left(x, x_0(t), u_0(t), \Pi_t, u_1(x,t)\right)\,p_{u_1}(x,t)\,dx\,dt\right.\\
    &\quad\left.+ \int_{\mathbb{R}^{n_1}}
    h_1\!\left(x, x_0(T), u_0(T), \Pi_T\right)\,p_{u_1}(x,T)\,dx \right].
\end{align*}

\section{Problems}

In this section we reformulate the optimization problems of the representative agent and the dominating player within the extended framework, in which the cost functions depend on the joint law of states and controls.

\subsection{Problem 1: Control of the Representative Agent}
\label{Problem1}
Given the process \(x_0\), the control $u_0$, and an exogenous flow of joint measures \(\Pi=(\Pi_t)_{t\in[0,T]}\), find a control \(u_1 \in \mathcal{A}_1\) that minimizes the cost functional
\begin{equation}
    \label{Problem 21}
    \begin{aligned}
    J_1(u_1; x_0, u_0, \Pi)
    := \mathbb{E}\left[\int_0^T f_1\big(x_1(t), x_0(t), u_0(t), \Pi_t, u_1(t)\big)\,dt
    + h_1\big(x_1(T),x_0(T), u_0(T), \Pi_T\big)\right].
    \end{aligned}
\end{equation}

\subsection{Problem 2: Equilibrium Condition}
\label{Problem2}
Let \(x_1^{u_1}\) be the dynamics induced by an optimal control \(u_1\) of \hyperref[Problem1]{Problem 1}. We denote by \(\mathcal{M}(\Pi)\) the joint law induced by the pair \((x_1^{u_1},u_1)\), conditioned on the filtration \(\mathcal{F}^0_t\). The equilibrium problem consists in finding a flow of measures \(\Pi\) such that the fixed-point property is satisfied:
\[
\mathcal{M}(\Pi)(t) = \Pi_t, \quad \forall\,t\in[0,T].
\]

\subsection{Problem 3: Control of the Dominating Player}
\label{Problem3}
Given the solution \(\Pi\) obtained in \hyperref[Problem2]{Problem 2}, find a control \(u_0\in\mathcal{A}_0\) that minimizes the cost functional
\begin{equation}
    \label{Problema 23}
    \begin{aligned}
    J_0(u_0;\Pi)
    := \mathbb{E}\left[\int_0^T f_0\big(x_0(t), \Pi_t, u_0(t)\big)\,dt
    + h_0\big(x_0(T),\Pi_T\big)\right].
    \end{aligned}
\end{equation}

\section{Main results}

\begin{lema}[Necessary condition for Problem 1]
    \label{lema:condicion_necesaria_problema1}
    Given \( x_0, u_0 \) and an exogenous flow of joint measures \(\Pi\) as in \hyperref[Problem1]{Problem 1}, the optimal control \( \hat{u}_1 \in \mathcal{A}_1 \) is optimal if and only if it satisfies the following SHJB equation:

\[
\left\{
\begin{aligned}
    -\partial_t \Psi &= \left( H_1(x, x_0(t), u_0(t), \Pi_t, D\Psi(x,t)) - A_1 \Psi(x,t) \right) dt - K_\Psi(x,t) dW_0(t), \\
    \Psi(x,T) &= h_1 \left( x, x_0(T), u_0(T), \Pi_T \right),
\end{aligned}
\right.
\]

where
\[
H_1(x, x_0, u_0, \Pi, q) = \inf_{u_1} \Big\{ f_1(x, x_0, u_0, \Pi, u_1) + q \, g_1(x, x_0, u_0, \Pi, u_1) \Big\}.
\]

The infimum is attained uniquely at \( \hat{u}_1 \), that is,
\[
H_1(x, x_0, u_0, \Pi, q) = f_1(x,x_0,u_0,\Pi,\hat{u}_1) + q\, g_1(x,x_0,u_0,\Pi,\hat{u}_1).
\]
\end{lema}

\begin{obs*}[Martingale process $K_\Psi$]
The process $K_\Psi(x,t)$ appearing in the SHJB equation is the coefficient of the martingale part in the Doob--Meyer decomposition of $\Psi$ with respect to the filtration $\mathcal{F}^0$ generated by the Brownian motion of the dominating player $W_0$. More precisely, the value function $\Psi(x,t)$ satisfies a BSDE of the form
$$-d\Psi(x,t) = F(x,t,D\Psi)\,dt - K_\Psi(x,t)\,dW_0(t),$$
where $K_\Psi$ is the square-integrable integrand that ensures the adaptedness of the solution. The existence of $K_\Psi$ is guaranteed by the martingale representation theorem in the Brownian filtration $\mathcal{F}^0$.
\end{obs*}

\begin{proof}
We apply the stochastic maximum principle. For any perturbation \(\tilde{u}_1 \in \mathcal{A}_1\),
\begin{equation}
\label{Condicion optimalidad}
    \begin{aligned}
    0 &= \left.\frac{d}{d\theta}\right|_{\theta=0} J_1(\hat{u}_1+\theta\tilde{u}_1,x_0,u_0,\Pi).
    \end{aligned}
\end{equation}
Rewriting the expectation in terms of the conditional density $p_{\hat{u}_1}(x,t)$ of the state $x_1$ and differentiating, we obtain
\begin{align*}
0 &= \mathbb{E}\Bigg[\int_0^T\int_{\mathbb{R}^{n_1}} \tilde{p}(x,t)\, f_1(x,x_0(t),u_0(t),\Pi_t,\hat{u}_1(x,t)) \, dx\,dt \\
&\quad + \int_0^T\int_{\mathbb{R}^{n_1}} p_{\hat{u}_1}(x,t)\, f_{1,u_1}(x,x_0(t),u_0(t),\Pi_t,\hat{u}_1(x,t)) \, \tilde{u}_1(x,t)\, dx\,dt \\
&\quad + \int_{\mathbb{R}^{n_1}} \tilde{p}(x,T)\, h_1(x,x_0(T),u_0(T),\Pi_T)\, dx \Bigg],
\end{align*}
where \(\tilde{p}=\left.\tfrac{d}{d\theta}\right|_{\theta=0} p_{\hat{u}_1+\theta\tilde{u}_1}\).

Differentiating with respect to $\theta$ in the Fokker--Planck equation \eqref{Fokker-Planck}, $\tilde{p}$ satisfies
    \begin{equation*}
        \begin{cases}
            \begin{aligned}
                \frac{\partial \tilde{p}}{\partial t} &= -A_1^*\tilde{p}(x,t)-\mathrm{div}\left(\tilde{u}_1(x,t)g_{1,u_1}\left(x,x_0(t),u_0(t),\Pi_t,\hat{u}_1(x,t)\right)p_{\hat{u}_1}(x,t)\right)\\
                &\quad - \mathrm{div}\left(g_1\left(x,x_0(t),u_0(t),\Pi_t,\hat{u}_1(x,t)\right)\tilde{p}(x,t)\right),\\
                \tilde{p}(x,0) &= 0.
            \end{aligned}
        \end{cases}
    \end{equation*}

We introduce the adjoint process $\Psi$ as the solution of the BSDE
    \[
        \left\{
        \begin{aligned}
            -\partial_t \Psi &= \Big( f_1(x, x_0(t), u_0(t), \Pi_t, u_1) +  D\Psi(x,t)\, g_1(x,x_0(t),u_0(t),\Pi_t,u_1) - A_1 \Psi(x,t) \Big)\, dt \\
            &\quad - K_\Psi(x,t)\, dW_0(t), \\
            \Psi(x,T) &= h_1 \left( x, x_0(T), u_0(T), \Pi_T \right).
        \end{aligned}
        \right.
    \]

Consider the inner product
    \begin{align*}
        & d\int_{\mathbb{R}^{n_1}}\tilde{p}(x,t)\Psi(x,t)dx \\
        &= \int_{\mathbb{R}^{n_1}} \left\{-A_1^*\tilde{p}(x,t)-\mathrm{div}\left[\tilde{u}_1(x,t)g_{1,u_1}(x,x_0(t),u_0(t),\Pi_t,\hat{u}_1(x,t))p_{\hat{u}_1}(x,t)\right]\right.\\
        &\hspace{2cm} \left. -\mathrm{div}\left[g_1(x,x_0(t),u_0(t),\Pi_t,\hat{u}_1(x,t))\tilde{p}(x,t)\right] \right\}\Psi(x,t)dxdt\\
        &\hspace{1cm} -\int_{\mathbb{R}^{n_1}} \tilde{p}(x,t)\left\{ f_1(x, x_0(t), u_0(t), \Pi_t, u_1(x,t)) \right.\\
        &\hspace{1cm}+\left. D\Psi(x,t)g_1(x,x_0(t),u_0(t),\Pi_t,u_1(x,t)) - A_1 \Psi(x,t) \right\} dxdt \\
        &\hspace{1cm} + \int_{\mathbb{R}^{n_1}} \tilde{p}(x,t)K_\Psi(x,t) dxdW_0(t)\\
        &=\int_{\mathbb{R}^{n_1}} \left( \tilde{u}_1(x,t)g_{1,u_1}(x,x_0(t),u_0(t),\Pi_t,\hat{u}_1(x,t))p_{\hat{u}_1}(x,t) \right)D\Psi(x,t)dxdt\\
        & -\int_{\mathbb{R}^{n_1}}\tilde{p}(x,t) f_1(x, x_0(t), u_0(t), \Pi_t, u_1(x,t)) dxdt+ \int_{\mathbb{R}^{n_1}} \tilde{p}(x,t)K_\Psi(x,t) dxdW_0(t).
    \end{align*}

Integrating over $[0,T]$ and taking expectations on both sides we get
    \begin{align*}
        & \mathbb{E}\left[\int_{\mathbb{R}^{n_1}}\tilde{p}(x,T)h_1(x,x_0(T),u_0(T),\Pi_T)dx\right] \\
        &=\mathbb{E}\left[\int_0^T\int_{\mathbb{R}^{n_1}}\left( \tilde{u}_1(x,t)g_{1,u_1}(x,x_0(t),u_0(t),\Pi_t,\hat{u}_1(x,t))p_{\hat{u}_1}(x,t) \right)D\Psi(x,t)dxdt \right.\\
        &\hspace{1cm} \left.-\int_0^T\int_{\mathbb{R}^{n_1}}\tilde{p}(x,t) f_1(x, x_0(t), u_0(t), \Pi_t, u_1(x,t)) dxdt\right],
    \end{align*}
where the term involving $K_\Psi$ vanishes upon taking expectations since it is a martingale term with zero expectation. Combining this with equation \eqref{Condicion optimalidad}, we obtain
    \begin{align*}
        0 &= \mathbb{E}\left[\int_0^T\int_{\mathbb{R}^{n_1}} \tilde{u}_1(x,t)\left[g_{1,u_1}(x,x_0(t),u_0(t),\Pi_t,\hat{u}_1(x,t))D\Psi(x,t) \right. \right. \\
        &+ \left. \left. f_{1,u_1}(x,x_0(t),u_0(t),\Pi_t,\hat{u}_1(x,t))\right] p_{\hat{u}_1}(x,t)dxdt\right].
    \end{align*}
Recall that $p_{\hat{u}_1}(\cdot,t)$ is a conditional probability density function and hence nonnegative, and $\tilde{u}_1$ is an arbitrary Markovian control. Therefore $\hat{u}_1$ is optimal only if
    $$g_{1,u_1}(x,x_0(t),u_0(t),\Pi_t,\hat{u}_1(x,t))D\Psi(x,t)+ f_{1,u_1}(x,x_0(t),u_0(t),\Pi_t,\hat{u}_1(x,t))=0,\quad a.e.\,(x,t).$$
This is equivalent to
    $$L_{u_1}(x,x_0(t),u_0(t),\Pi_t,\hat{u}_1(x,t), D\Psi(x,t))=0,\quad a.e.\,(x,t),$$
which provides a necessary condition for the minimization problem. Since the minimizer is assumed to be attained at $\hat{u}_1$, which depends on $x, x_0, u_0, \Pi$ and $D\Psi$, we obtain the SHJB equation.
\end{proof}

We replace the exogenous flow $\Pi$ by the mean field measure $\Pi^{x_0,u_0}$, identifying $m_{x_0,u_0}:=p_{\hat{u}_1}\,d\lambda$ with the conditional density of the optimal state of the representative agent given $\mathcal{F}_t^0$. Combining equations \eqref{Fokker-Planck} and \eqref{Condicion optimalidad} we obtain the following corollary.

\begin{cor}[Necessary condition for Problems 1 and 2]
\label{Condicion necesaria Problemas 1 y 2}
The control for the representative agent is optimal and the equilibrium condition is satisfied if and only if the coupled SHJB--FP system holds:

\[
\left\{
\begin{aligned}
    -\partial_t \Psi &= \Big( H_1(x, x_0(t), u_0(t), \Pi^{x_0,u_0}_t, D\Psi(x,t)) - A_1 \Psi(x,t) \Big) dt - K_\Psi(x,t) \, dW_0(t), \\
    \Psi(x,T) &= h_1 \left( x, x_0(T), u_0(T), \Pi^{x_0,u_0}_T \right), \\
    \partial_t m_{x_0,u_0} &= -A_1^* m_{x_0,u_0}(x,t) - \mathrm{div}\!\Big( G_1(x, x_0(t), u_0(t), \Pi^{x_0,u_0}_t, D\Psi(x,t))\, m_{x_0,u_0}(x,t)\Big), \\
    m_{x_0,u_0}(x,0) &= \omega(x),
\end{aligned}
\right.
\]

where:
\[
G_1(x,x_0,u_0,\Pi,q) = g_1\!\big(x, x_0, u_0,\Pi, \hat u_1(x, x_0, u_0, \Pi, q)\big),
\]
$\Pi^{x_0,u_0}$ denotes the state--control law and \(m_{x_0,u_0}(x,t)\) denotes the conditional density of \(x_1(t)\), given \(x_0,u_0\).

In particular, the passage from Problem 1 to Problem 2 is interpreted as:
\[
\Pi^{x_0,u_0}_t = \mathcal{L}\!\left( x_1(t), \hat u_1(x_1(t),x_0(t), u_0(t),\Pi,D\Psi)\,\big|\, \mathcal{F}_t^0 \right),
\]
that is, the joint measure induced by the optimal state and control of the representative agent conditioned on the information of the dominating player.
\end{cor}

\section{Extension of the results}

\begin{prop}[Necessary condition for Problem 3 in the extended case]
\label{prop:condicion_necesaria_problema3}
Let \(\Pi^{x_0,u_0}_t=\mathcal{L}\big(x_1(t),\hat u_1(t)\,\big|\,\mathcal{F}^0_t\big)\in\mathcal{P}_2(\mathbb{R}^{n_1}\times A)\) be the joint law induced by the representative agent (Problems 1--2). The control of the dominating player \(\hat u_0\in\mathcal{A}_0\) is optimal if and only if
\[
f_0\!\big(x_0,\Pi^{x_0,u_0}_t,\hat u_0(t)\big)+p(t)\cdot g_0\!\big(x_0,\Pi^{x_0,u_0}_t,\hat u_0(t)\big)=\inf_{u_0\in \mathcal{A}_0}\Big\{f_0(x_0,\Pi^{x_0,u_0}_t,u_0)+p(t)\cdot g_0(x_0,\Pi^{x_0,u_0}_t,u_0)\Big\}
\]
We set $H_0\!\big(x_0,\Pi^{x_0,u_0}_t,p(t)\big):=\inf_{u_0\in \mathcal{A}_0}\Big\{f_0(x_0,\Pi^{x_0,u_0}_t,u_0)+p(t)\cdot g_0(x_0,\Pi^{x_0,u_0}_t,u_0)\Big\}$. The adjoint process \(p\) and the adjoint fields \(q(\cdot,t),r(\cdot,t)\) satisfy, for \(t\in[0,T]\), the system

\[
\begin{aligned}
\label{Adj-p}
-dp(t)&=\bigg[
g_{0,x_0}\!\big(x_0(t),\Pi^{x_0,u_0}_t,\hat u_0(t)\big)\,p(t)
+f_{0,x_0}\!\big(x_0(t),\Pi^{x_0,u_0}_t,\hat u_0(t)\big) \\
&\qquad\ +\int_{\mathbb{R}^{n_1}} \! G_{1,x_0}\!\big(x,x_0(t),u_0(t),\Pi^{x_0,u_0}_t,D\Psi(x,t)\big)\cdot Dq(x,t)\, m_{x_0,u_0}(x,t)\,dx \\
&\qquad\ +\int_{\mathbb{R}^{n_1}} \! r(x,t)\, H_{1,x_0}\!\big(x,x_0(t),u_0(t),\Pi^{x_0,u_0}_t,D\Psi(x,t)\big)\,dx
\bigg]dt \\
&\qquad -\sum_{\ell=1}^{d_0} K_p^\ell(t)\,dW_0^\ell(t)
+\sum_{\ell=1}^{d_0}\sigma^{\ell_*}_{0,x_0}\!\big(x_0(t)\big)\,K_p^\ell(t)\,dt, \\[0.3em]
p(T)&=h_{0,x_0}\!\big(x_0(T),\Pi^{x_0,u_0}_T\big)+\int_{\mathbb{R}^{n_1}} r(x,T)\,h_{1,x_0}\!\big(x,x_0(T),u_0(T),\Pi^{x_0,u_0}_T\big)\,dx,
\end{aligned}
\tag{Adj-p}
\]

\[
\begin{aligned}
\label{Adj-q}
-\partial_t q(x,t)&=\Big[-A_1 q(x,t)
+p(t)\,\Big\langle D_{\Pi}g_0\!\big(x_0(t),\Pi^{x_0,u_0}_t,\hat u_0(t)\big)\,;\,\delta_x\Big\rangle \\
&\quad +Dq(x,t)\cdot G_1\!\big(x,x_0(t),u_0(t),\Pi^{x_0,u_0}_t,D\Psi(x,t)\big) \\
&\quad +\int_{\mathbb{R}^{n_1}\times A} Dq(\xi,t)\cdot
\Big\langle D_{\Pi}G_1\!\big(\xi,x_0(t),u_0(t),\Pi^{x_0,u_0}_t,D\Psi(\xi,t)\big)\,;\,\delta_x\Big\rangle \,\Pi^{x_0,u_0}_t(d\xi,du) \\
&\quad +\int_{\mathbb{R}^{n_1}\times A} r(\xi,t)\,
\Big\langle D_{\Pi}H_1\!\big(\xi,x_0(t),u_0(t),\Pi^{x_0,u_0}_t,D\Psi(\xi,t)\big)\,;\,\delta_x\Big\rangle \,\Pi^{x_0,u_0}_t(d\xi,du) \\
&\quad +\Big\langle D_{\Pi}f_0\!\big(x_0(t),\Pi^{x_0,u_0}_t,\hat u_0(t)\big)\,;\,\delta_x\Big\rangle \Big] dt
\ -\ K_q(x,t)\,dW_0(t), \\
q(x,T)&=\Big\langle D_{\Pi}h_0\!\big(x_0(T),\Pi^{x_0,u_0}_T\big)\,;\,\delta_x\Big\rangle
+\int_{\mathbb{R}^{n_1}\times A} r(\xi,T)\,
\Big\langle D_{\Pi}h_1\!\big(\xi,x_0(T),u_0(T),\Pi^{x_0,u_0}_T\big)\,;\,\delta_x\Big\rangle \,\Pi^{x_0,u_0}_T(d\xi,du),
\end{aligned}
\tag{Adj-q}
\]

\[
\begin{aligned}
\label{Adj-r}
\partial_t r(x,t)&=-A_1^* r(x,t)\ -\ \mathrm{div}_x\!\Big(
r(x,t)\,H_{1,q}\!\big(x,x_0(t),u_0(t),\Pi^{x_0,u_0}_t,D\Psi(x,t)\big) \\
&\quad +\ G_{1,q}\!\big(x,x_0(t),u_0(t),\Pi^{x_0,u_0}_t,D\Psi(x,t)\big)\,Dq(x,t)\,m_{x_0,u_0}(x,t)
\Big), \\
r(x,0)&=0.
\end{aligned}
\tag{Adj-r}
\]
\end{prop}

\begin{obs*}
In the adjoint equations, the derivative $G_{1,x_0}$ denotes the derivative of the drift evaluated at the optimal control:
$$G_{1,x_0} := \frac{\partial}{\partial x_0}g_1(\cdot, \hat{u}_1(\cdot)),$$
where $\hat{u}_1$ is determined from the SHJB equation of the representative agent. Analogously, $G_{1,q}$ denotes
$$G_{1,q} := \frac{\partial}{\partial q}g_1(\cdot, \hat{u}_1(\cdot)).$$
\end{obs*}

\begin{proof}[Sketch of proof]
The proof follows the scheme of Theorem 4.1 of Bensoussan, Chau and Yam \cite{bensoussan2014mfgdominating}, adapted to the dependence on the joint law $\Pi_t \in \mathcal{P}_2(\mathbb{R}^{n_1} \times A)$.

For clarity we first present the case in which the cost functions depend only on \(x_0\) (not on \(u_0\)); the extended case is discussed in the Remark at the end of this section.

\emph{(i) G\^{a}teaux variation.} We perturb \(u_0\mapsto \hat u_0+\theta\tilde u_0\) and differentiate in \(\theta\):
    \begin{equation}
        \label{Derivada Gateaux}
        \begin{aligned}
            0 &= \left.\frac{d}{d\theta}\right|_{\theta=0}J_0(\hat{u}_0+\theta\tilde{u}_0) \\
            &= \mathbb{E}\left\{ \int_0^T \left[f_{0,x_0}(x_0(t),\Pi^{x_0,u_0}_t,\hat{u}_0(t))\tilde{x}_0(t)+\int_{\mathbb{R}^{n_1}\times A} \Big\langle D_\Pi f_0(x_0(t),\Pi^{x_0,u_0}_t,\hat{u}_0(t));\delta_\xi\Big\rangle\,d\tilde\Pi_t(\xi,u) \right.\right.\\
            &\left.+ f_{0,u_0}(x_0(t),\Pi^{x_0,u_0}_t,\hat{u}_0(t))\tilde{u}_0(t)\right]dt \\
            &\left.+ h_{0,x_0}(x_0(T),\Pi^{x_0,u_0}_T)\tilde{x}_0(T)+\int_{\mathbb{R}^{n_1}\times A} \Big\langle D_\Pi h_0(x_0(T),\Pi^{x_0,u_0}_T);\delta_\xi\Big\rangle\,d\tilde\Pi_T(\xi,u) \right\},
        \end{aligned}
    \end{equation}
where $\tilde{x}_0=\left.\frac{d}{d\theta}\right|_{\theta=0}x_0(\hat{u}_0+\theta\tilde{u}_0)$; $\tilde{m}_{x_0,u_0}=\left.\frac{d}{d\theta}\right|_{\theta=0}m_{x_0,u_0}(\hat{u}_0+\theta\tilde{u}_0)$; $\tilde{\Psi}=\left.\frac{d}{d\theta}\right|_{\theta=0}\Psi(\hat{u}_0+\theta\tilde{u}_0)$.

\emph{(ii) Linearized equations.} The variations satisfy
    \begin{align*}
        d\tilde{x}_0 &= \left[ g_{0,x_0}(x_0(t), \Pi^{x_0,u_0}_t, \hat{u}_0(t)) \tilde{x}_0(t) + \int_{\mathbb{R}^{n_1}\times A} \Big\langle D_\Pi g_0(x_0(t), \Pi^{x_0,u_0}_t, \hat{u}_0(t));\delta_\xi\Big\rangle\,d\tilde\Pi_t(\xi,u) \right.\\
        & + \left. g_{0,u_0}(x_0(t), \Pi^{x_0,u_0}_t, \hat{u}_0(t)) \tilde{u}_0(t)\right]dt + \sum_{l=1}^{d_0} \sigma^l_{0,x_0}\tilde{x}_0(t)dW_0^l(t); \\
        \tilde{x}_0(0) &= 0; \\
        \frac{\partial \tilde{m}_{x_0,u_0}}{\partial t} &= -A_1^* \tilde{m}_{x_0,u_0}(x,t) - \mathrm{div}\left\{\left[ G_{1,x_0}(x, x_0(t), u_0(t), \Pi^{x_0,u_0}_t, D\Psi(x,t))\tilde{x}_0(t) \right. \right. \\
        &+ \int_{\mathbb{R}^{n_1}\times A} \Big\langle D_\Pi G_1(x,x_0(t),u_0(t),\Pi^{x_0,u_0}_t,D\Psi(x,t));\delta_\xi\Big\rangle\,d\tilde\Pi_t(\xi,u) \\
        & \left. + G_{1,q}(x,x_0(t),u_0(t),\Pi^{x_0,u_0}_t,D\Psi(x,t))D\tilde{\Psi}(x,t)\right]m_{x_0,u_0}(x,t)\\
        & \left. + G_1(x,x_0(t),u_0(t),\Pi^{x_0,u_0}_t,D\Psi(x,t))\tilde{m}_{x_0,u_0}(x,t)\right\}, \\
        \tilde{m}_{x_0,u_0}(x,0) &= 0; \\
        -\partial_t \tilde{\Psi} &= \left[ H_{1,x_0}(x, x_0(t), u_0(t), \Pi^{x_0,u_0}_t, D\Psi(x,t))\tilde{x}_0(t) \right. \\
        &+ \int_{\mathbb{R}^{n_1}\times A} \Big\langle D_\Pi H_1(x,x_0(t),u_0(t),\Pi^{x_0,u_0}_t,D\Psi(x,t));\delta_\xi\Big\rangle\,d\tilde\Pi_t(\xi,u) \\
        & \left. + H_{1,q}(x,x_0(t),u_0(t),\Pi^{x_0,u_0}_t,D\Psi(x,t))D\tilde{\Psi}(x,t)-A_1\tilde{\Psi}(x,t)\right]dt-K_{\tilde{\Psi}}(x,t)dW_0(t),\\
        \tilde{\Psi}(x,T) &= h_{1,x_0}(x, x_0(T),u_0(T),\Pi^{x_0,u_0}_T)\tilde{x}_0(T)+\int_{\mathbb{R}^{n_1}\times A} \Big\langle D_\Pi h_1(x,x_0(T),u_0(T),\Pi^{x_0,u_0}_T);\delta_\xi\Big\rangle\,d\tilde\Pi_T(\xi,u).
    \end{align*}

\emph{(iii) Adjoint coupling.} We introduce the adjoint processes $p(t)$, $q(x,t)$ and $r(x,t)$ as in the Proposition and consider the dual products
\[
d\big(p^* \tilde x_0\big),\qquad d\!\int q(x,t)\,\tilde m_{x_0,u_0}(x,t)\,dx,\qquad d\!\int r(x,t)\,\tilde\Psi(x,t)\,dx.
\]
Integrating by parts and using that $\Pi^{x_0,u_0}_t$ appears only through $D_\Pi f_0, D_\Pi g_0, D_\Pi H_1, D_\Pi G_1$, all terms in $\tilde x_0, \tilde m_{x_0,u_0}, \tilde\Psi$ cancel by the choice of $(p,q,r)$ solving \eqref{Adj-p}--\eqref{Adj-q}--\eqref{Adj-r}. Summing the three differentials and taking expectations we obtain
\begin{align*}
    & d\{p^*\tilde{x}_0+\int q(x,t)\tilde{m}_{x_0,u_0}(x,t)dx-\int r(x,t)\tilde{\Psi}(x,t)dx\} \\
    &= \left[-\tilde{x}_0^* f_{0,x_0}(x_0(t), \Pi^{x_0,u_0}_t, \hat{u}_0(t))+p^*(t)g_{0,u_0}(x_0(t), \Pi^{x_0,u_0}_t, \hat{u}_0(t)) \tilde{u}_0(t) \right.\\
    &\quad - \left.\int_{\mathbb{R}^{n_1}\times A} \Big\langle D_\Pi f_0(x_0(t),\Pi^{x_0,u_0}_t,\hat{u}_0(t));\delta_x\Big\rangle\,d\tilde\Pi_t(x,u)\right]dt+\{\cdots\}dW_0(t).
\end{align*}
Integrating and taking expectations on each side, and considering \eqref{Derivada Gateaux}, we observe that the terms with $r$ cancel each other, leaving
$$0=\mathbb{E}\int_0^T\left\{f_{0,u_0}(x_0,\Pi^{x_0,u_0}_t,\hat{u}_0)+p^*g_{0,u_0}(x_0,\Pi^{x_0,u_0}_t,\hat{u}_0)\right\}\tilde{u}_0\,dt.$$
Since $\tilde{u}_0$ is arbitrary, the control is optimal for the dominating player only if
$$f_{0,u_0}(x_0,\Pi^{x_0,u_0}_t,\hat{u}_0)+p^*g_{0,u_0}(x_0,\Pi^{x_0,u_0}_t,\hat{u}_0)=0,\quad a.e.\,t.$$
Since we are assuming that a unique minimizer $\hat{u}_0$ exists, we conclude that $\hat{u}_0$ satisfies the infimum condition
$$f_0(x_0,\Pi^{x_0,u_0}_t,\hat{u}_0)+p\cdot g_0(x_0,\Pi^{x_0,u_0}_t,\hat{u}_0)=\inf_{u_0}\left\{f_0(x_0,\Pi^{x_0,u_0}_t,u_0)+p\cdot g_0(x_0,\Pi^{x_0,u_0}_t,u_0)\right\}.$$

The rigorous verification that the Lions derivative $D_\Pi$ and the It\^{o} formula over flows of measures extend to the product space $\mathcal{P}_2(\mathbb{R}^{n_1} \times A)$ under hypotheses (A.1)--(A.5) is detailed in \cite{munoz2025complexmarkets}.
\end{proof}

\begin{obs*}
In this extended version, the adjoint processes retain their fundamental role, but their interpretation is broadened:

\begin{itemize}
    \item \bm{$p(t)$}: continues to be the costate of the dominating player, capturing the sensitivity of the functional with respect to variations in its state and control $x_0, u_0$.

    \item \bm{$q(x,t)$}: now encodes the sensitivity with respect to the extended measure $\Pi^{x_0,u_0}_t$. In practice, this means that $q$ captures both the variations with respect to the state distribution and those induced indirectly by the law of the controls, given by the marginals of $\Pi$. In this sense, $q$ absorbs the Lions derivatives with respect to both marginals and acts as the true multiplier for the consistency constraint on the population dynamics.

    \item \bm{$r(x,t)$}: as in the classical case, it remains the multiplier associated with the backward equation of the agents (the SHJB). In the extended version, its role is to ensure that the variations in $\Psi$ (which now depend on $\Pi$ and not only on $m_t$) are correctly transmitted to the system of necessary conditions.
\end{itemize}

Thus, the processes $(p,q,r)$ form a dual system that reflects the tripartite structure of the constraints: dynamics of the dominating player, evolution of the population, and consistency of the individual value. The key difference with respect to the classical case is that the sensitivity captured by $q$ unifies in a single object the dependencies on the state and on the control through the measure $\Pi^{x_0,u_0}_t$.
\end{obs*}

\begin{obs*}
When the functions $f_1,g_1$ also depend on the control of the dominating player $u_0$,
the optimality condition for the control becomes
\begin{align*}
    f_{0,u_0}(x_0,\Pi^{x_0,u_0}_t,u_0)&+p\cdot g_{0,u_0}(x_0,\Pi^{x_0,u_0}_t,u_0) \\
    &+\int_{\mathbb{R}^{n_1}} r(x,t)\,H_{1,u_0}\big(x,x_0(t),u_0(t),\Pi^{x_0,u_0}_t,D\Psi(x,t)\big)\,dx\\
    &+\int_{\mathbb{R}^{n_1}} G_{1,u_0}\big(x,x_0(t),u_0(t),\Pi^{x_0,u_0}_t,D\Psi(x,t)\big)Dq(x,t)\,m_{x_0,u_0}(x,t)\,dx=0,\quad a.e.\,t.
\end{align*}
This means that the first-order condition for $u_0$ is not obtained from the classical Hamiltonian $H_0$ but from an effective Hamiltonian that incorporates the reaction of the representative agents:
\begin{align*}
    H_0(x_0(t),\Pi^{x_0,u_0}_t;\, p(t),q(\cdot,t),r(\cdot,t))
    &:=\inf_{u_0\in \mathcal{A}_0} \left\{f_0(x_0,\Pi^{x_0,u_0}_t,u_0)+p\cdot g_0(x_0,\Pi^{x_0,u_0}_t,u_0)\right. \\
    &+\int_{\mathbb{R}^{n_1}} r(x,t)\,H_1\big(x,x_0(t),u_0(t),\Pi^{x_0,u_0}_t,D\Psi(x,t)\big)\,dx\\
    &\left.+\int_{\mathbb{R}^{n_1}} Dq(x,t)\,G_1\big(x,x_0(t),u_0(t),\Pi^{x_0,u_0}_t,D\Psi(x,t)\big)\,m_{x_0,u_0}(x,t)\,dx\right\}.
\end{align*}
The control of the dominating player $\hat{u}_0\in\mathcal{A}_0$ is then optimal if and only if
\[
0=\partial_{u_0}H_0(x_0(t),\Pi^{x_0,u_0}_t,\hat u_0(t);\,
p(t),q(\cdot,t),r(\cdot,t)),
\qquad \text{a.e. } t\in[0,T].
\]
Here:
\begin{itemize}
    \item $p$ is the adjoint process associated with the state of the dominating player $x_0$,
    \item $q$ and $r$ are the adjoint processes capturing the sensitivity of the representative agents' system,
    \item the terms with $H_{1,u_0}$ and $G_{1,u_0}$ describe the marginal impact of $u_0$ on the equilibrium of the representative agent (cost and dynamics respectively).
\end{itemize}
\end{obs*}

\begin{obs*}
In our extended formulation the cost functional depends on the joint measure
\(\Pi^{x_0,u_0}_t\) of the pair \((x,u)\). It might seem natural that, when computing the G\^{a}teaux derivative in the direction of a variation of the control, a variational process \(\tilde{\Pi}\) should appear describing the evolution of the perturbed measure. However, this does not occur, for two fundamental reasons:

\begin{enumerate}
    \item The Lions derivative with respect to \(\Pi\) already incorporates automatically the sensitivities in both coordinates: the state \(m_t\) and the induced distribution of controls \(\Lambda_{x_0}\). Thus, what would conceptually correspond to a term in \(\tilde{\Pi}\) decomposes into variations of \(\tilde{m}\) and \(\tilde{\Lambda}\), which are precisely those that appear in the computations.

    \item From the dynamical point of view, the control has no diffusion of its own: it is determined as a feedback from the trajectory of \(x_0\) and the common information. Consequently, in the evolution of \(\Pi\) factored as
    \[
        \Pi_t(dx,du) \;=\; m_t(dx)\,\Lambda_{x}(du),
    \]
    \(m_t\) is the only measure satisfying a genuine Fokker--Planck equation. The part associated with the control is obtained directly from the equilibrium policy without requiring an independent dynamics.
\end{enumerate}

In summary, the variation of the functional under perturbations of the control is expressed solely in terms of \(\tilde{m}\) and \(\tilde{\Lambda}\), which are absorbed upon introducing the adjoint process $q$, and the only Fokker--Planck equation that appears is that of \(m_t\). This explains why the computations retain the same structure as in the classical case.
\end{obs*}

Let $G_0(x_0,\Pi,p):= g_0(x_0,\Pi,\hat{u}_0(x_0,\Pi,p))$. We conclude with the main result of this work.

\begin{teo}[Extended necessary condition for Problems 1, 2 and 3]
\label{teo:CNO_extendida}
The system formed by the representative agent and the dominating player admits as necessary optimality condition the following coupled system of stochastic partial differential equations:

\[
\left\{
\begin{aligned}
    dx_0 &= G_0\big(x_0(t),\Pi^{x_0,u_0}_t, p(t)\big)\, dt \;+\; \sigma_0(x_0(t))\, dW_0(t), \\
    x_0(0)&=\xi_0;\\[0.5cm]
    \frac{\partial m_{x_0,u_0}}{\partial t} &= -A_1^* m_{x_0,u_0}(t)\;-\; \mathrm{div}\Big( G_1\big(x, x_0,u_0,\Pi^{x_0,u_0}_t, D\Psi(x,t)\big)\, m_{x_0,u_0}(t)\Big),\\
    m_{x_0,u_0}(x,0)&=\omega(x); \\[0.5cm]
    -\partial_t \Psi &= \Big( H_1\big(x, x_0(t),u_0(t),\Pi^{x_0,u_0}_t, D\Psi(x,t)\big) - A_1 \Psi(x,t) \Big)\,dt - K_\Psi(x,t)\, dW_0(t), \\
    \Psi(x,T) &= h_1\big(x, x_0(T),u_0(T),\Pi^{x_0,u_0}_T\big).
\end{aligned}
\right.
\]
\[
\left\{
\begin{aligned}
    -dp(t)&=\bigg[
    G_{0,x_0}\!\big(x_0(t),\Pi^{x_0,u_0}_t,\hat u_0(t)\big)\,p(t)
    +f_{0,x_0}\!\big(x_0(t),\Pi^{x_0,u_0}_t,\hat u_0(t)\big) \\
    &\qquad\ +\int_{\mathbb{R}^{n_1}} \! G_{1,x_0}\!\big(x,x_0(t),u_0(t),\Pi^{x_0,u_0}_t,D\Psi(x,t)\big)\cdot Dq(x,t)\, m_{x_0,u_0}(x,t)\,dx \\
    &\qquad\ +\int_{\mathbb{R}^{n_1}} \! r(x,t)\, H_{1,x_0}\!\big(x,x_0(t),u_0(t),\Pi^{x_0,u_0}_t,D\Psi(x,t)\big)\,dx
    \bigg]dt \\
    &\qquad -\sum_{\ell=1}^{d_0} K_p^\ell(t)\,dW_0^\ell(t)
    +\sum_{\ell=1}^{d_0}\sigma^{\ell_*}_{0,x_0}\!\big(x_0(t)\big)\,K_p^\ell(t)\,dt, \\[0.3em]
    p(T)&=h_{0,x_0}\!\big(x_0(T),\Pi^{x_0,u_0}_T\big)+\int_{\mathbb{R}^{n_1}} r(x,T)\,h_{1,x_0}\!\big(x,x_0(T),u_0(T),\Pi^{x_0,u_0}_T\big)\,dx; \\[0.5cm]
    -\partial_t q&=\Big[-A_1 q(x,t)
    +p(t)\,\Big\langle D_{\Pi}G_0\!\big(x_0(t),\Pi^{x_0,u_0}_t,\hat u_0(t)\big)\,;\,\delta_x\Big\rangle \\
    &\quad +Dq(x,t)\cdot G_1\!\big(x,x_0(t),u_0(t),\Pi^{x_0,u_0}_t,D\Psi(x,t)\big) \\
    &\quad +\int_{\mathbb{R}^{n_1}\times A} Dq(\xi,t)\cdot
    \Big\langle D_{\Pi}G_1\!\big(\xi,x_0(t),u_0(t),\Pi^{x_0,u_0}_t,D\Psi(\xi,t)\big)\,;\,\delta_x\Big\rangle \,\Pi^{x_0,u_0}_t(d\xi,du) \\
    &\quad +\int_{\mathbb{R}^{n_1}\times A} r(\xi,t)\,
    \Big\langle D_{\Pi}H_1\!\big(\xi,x_0(t),u_0(t),\Pi^{x_0,u_0}_t,D\Psi(\xi,t)\big)\,;\,\delta_x\Big\rangle \,\Pi^{x_0,u_0}_t(d\xi,du) \\
    &\quad +\Big\langle D_{\Pi}f_0\!\big(x_0(t),\Pi^{x_0,u_0}_t,\hat u_0(t)\big)\,;\,\delta_x\Big\rangle \Big] dt
    \ -\ K_q(x,t)\,dW_0(t), \\
    q(x,T)&=\Big\langle D_{\Pi}h_0\!\big(x_0(T),\Pi^{x_0,u_0}_T\big)\,;\,\delta_x\Big\rangle
    +\int_{\mathbb{R}^{n_1}\times A} r(\xi,T)\,
    \Big\langle D_{\Pi}h_1\!\big(\xi,x_0(T),u_0(T),\Pi^{x_0,u_0}_T\big)\,;\,\delta_x\Big\rangle \,\Pi^{x_0,u_0}_T(d\xi,du);\\[0.5cm]
    \partial_t r&=-A_1^* r(x,t)\ -\ \mathrm{div}_x\!\Big(
    r(x,t)\,H_{1,q}\!\big(x,x_0(t),u_0(t),\Pi^{x_0,u_0}_t,D\Psi(x,t)\big) \\
    &\qquad\qquad\qquad\quad +\ G_{1,q}\!\big(x,x_0(t),u_0(t),\Pi^{x_0,u_0}_t,D\Psi(x,t)\big)\,Dq(x,t)\,m_{x_0,u_0}(x,t)
    \Big), \\
    r(x,0)&=0.
\end{aligned}
\right.
\]
\end{teo}

\begin{obs*}
In the adjoint equations, the derivative $G_{0,x_0}$ denotes the derivative of the drift evaluated at the optimal control:
$$G_{0,x_0} := \frac{\partial}{\partial x_0}g_0(x_0, \Pi, \hat{u}_0(x_0, \Pi, p)),$$
where $\hat{u}_0$ is the solution of $\partial_{u_0} H_0 = 0$.
\end{obs*}

\begin{obs*}
The system of equations obtained constitutes the extended version of the classical adjoint SHJB--FP scheme. Indeed, it combines:

\begin{itemize}
    \item The dynamics of the dominating state \(x_0\), governed by a control \(u_0\) that interacts with the joint law of states and controls \(\Pi^{x_0,u_0}=(m_{x_0,u_0}, \Lambda_{x_0,u_0})\).
    \item The Fokker--Planck equation for the density of the representative agent, which describes the aggregate evolution of the population in terms of the conditioned measure \(m_{x_0,u_0}\).
    \item The stochastic Hamilton--Jacobi--Bellman equation for the representative agent, whose dependence on \(\Pi\) explicitly reflects the presence of the law of the states and of the controls in the cost function.
    \item The adjoint processes \(p(t), q(x,t), r(x,t)\) associated with the problem of the dominating player, which encode the derivatives of \(f_0,h_0\) with respect to \(x_0\), the state measure, and the law of the controls.
\end{itemize}

This coupled system simultaneously synthesizes the optimality conditions for the representative agent (Problem 1), the fixed-point equilibrium condition (Problem 2), and the optimality condition for the dominating player (Problem 3). In analytic terms, it is a Forward--Backward Stochastic PDE system that generalizes previous approaches to Mean Field Games, explicitly incorporating the dependence on the extended law \(\Pi\).
\end{obs*}

\section{Conclusion}

In this work we extended the theory of mean field games with a dominating player of Bensoussan, Chau and Yam \cite{bensoussan2014mfgdominating} to the case where the cost functionals and dynamics depend on the joint state--control law $\Pi_t = \mathcal{L}(X_t^1, U_t^1 \mid \mathcal{F}_t^0) \in \mathcal{P}_2(\mathbb{R}^{n_1} \times A)$, rather than solely on the marginal distribution of states $\mu_t$.

The main contributions are threefold.

\paragraph{Reformulation in terms of $\Pi_t$.}
The optimization problems of the representative agent (Problems 1 and 2) and of the dominating player (Problem 3) were reformulated in terms of the joint measure $\Pi_t$. This unifies the dependencies on state and control into a single measure on the product space $\mathbb{R}^{n_1} \times A$, simplifying both the notation and the application of differential calculus on probability spaces.

\paragraph{Extension of the Lions derivative.}
The Lions derivative extends naturally to the space $\mathcal{P}_2(\mathbb{R}^{n_1} \times A)$ when $A \subset \mathbb{R}^m$ is compact, leveraging the product space structure. This extension is essential for formulating the optimality conditions, as it allows one to compute derivatives of functionals with respect to the measure $\Pi$ and decompose them into contributions from the state and control marginals.

\paragraph{Coupled SHJB--FP--adjoint system.}
Proposition~\ref{prop:condicion_necesaria_problema3} establishes the necessary optimality conditions for the dominating player through a system of three adjoint processes $(p, q, r)$, and Theorem~\ref{teo:CNO_extendida} synthesizes the equilibrium conditions of the full system into a coupled Forward--Backward stochastic--partial differential equation system. This result generalizes Theorem~4.1 of \cite{bensoussan2014mfgdominating} to the new measure space. The dependence on $\Pi$ (rather than on $\mu$ alone) modifies the adjoint equation for $q(x,t)$, which now captures simultaneously the sensitivities with respect to both state and control through the joint measure.

\paragraph{Future directions.}
The extension presented here opens several lines of investigation. First, the question of existence and uniqueness of solutions for the extended SHJB--FP--adjoint system requires deeper analysis, particularly regarding the regularity of the Lions derivative on the product space. Second, the application of this framework to concrete models---such as decentralized markets where a liquidity provider acts as a dominating player facing a continuum of traders \cite{munoz2025complexmarkets}---constitutes a natural motivation and a test of the generality of the theory. Finally, the extension to games where multiple dominating players interact with each other and with the field of minor agents is an open problem of both theoretical and applied interest.


\bibliographystyle{plain}
\bibliography{main}

@article{carmona2015weakmfgformulation,
  title={A probabilistic weak formulation of mean field games and applications},
  author={R. Carmona and D. Lacker},
  journal={The Annals of Applied Probability},
  volume={25},
  number={3},
  pages={1189-1231},
  year={2015}
}

@article{bensoussan2014mfgdominating,
  title={Mean Field Games with a Dominating Player},
  author={A. Bensoussan M.H.M. Chau and S.C.P. Yam},
  journal={},
  volume={},
  number={},
  pages={},
  year={2014},
  publisher={}
}

@misc{guo2022itoformula,
  title={Ito's formula for flow of measures on semimartingales},
  author={X. Guo and C. Liang},
  year={2022},
  eprint={2208.01572},
  archivePrefix={arXiv}
}

@misc{korici2024lionsderivative,
  title={Lions derivative and its applications},
  author={S. Korici},
  year={2024},
  note={Lecture notes}
}

@article{munoz2025complexmarkets,
  title={Mean Field Games in Complex Markets: Extensions and Applications},
  author={A. Muñoz Gonzalez},
  journal={},
  volume={},
  number={},
  pages={},
  year={2025},
  publisher={}
}
\end{document}